\documentclass[12pt]{amsart}

\usepackage{bbm}
\usepackage{mathrsfs}
\usepackage{amsfonts}
\usepackage{amsmath}
\usepackage{amssymb}
\usepackage{graphicx}
\usepackage{latexsym}
\usepackage{cases}
\usepackage{amscd}
\usepackage{amsfonts}
\usepackage[all]{xy}
\usepackage{xy}
\usepackage{amsfonts,amsmath, amssymb}

\newtheorem{theorem}{Theorem}[section]
\newtheorem{corollary}[theorem]{Corollary}

\newtheorem{proposition}[theorem]{Proposition}
\newtheorem{remark}[theorem]{Remark}
\newtheorem{definition}[theorem]{Definition}

\begin{document}
\title{Nilpotent Elements of Vertex Algebras}

\author {Xianzu Lin }

\date{ }
\maketitle
  {\small \it College of Mathematics and Computer Science}\\
 {\small \it Fujian Normal University, Fuzhou, 350108,
      China}\\
               {\small \it Email: linxianzu@126.com}

\begin{abstract}
Using the method of commutative algebra, we show that the set
$\mathfrak{R}$ of nilpotent elements of a vertex algebra $V$ forms
an ideal, and $V/\mathfrak{R}$ has no nonzero nilpotent elements.

\

Keywords: nilpotent element; prime ideal; vertex algebra

\ 2000 MR Subject Classification:17B69, 13A15
\end{abstract}

\section{Introduction}
In the development of the axiomatic theory of vertex algebra in
\cite{ll}, its analogy with the classical theory of commutative
rings was greatly emphasized. In this paper we further explore the
analogy between vertex algebras and commutative rings, and the main
result is the following.

\begin{theorem}\label{th1}
The set $\mathfrak{R}$ of nilpotent elements of a vertex algebra $V$
forms an ideal, and $V/\mathfrak{R}$ has no nonzero nilpotent
elements.
\end{theorem}

\begin{remark}
The first assertion was proved in \S3.10 of \cite{ll}. However,
their method was a bit complicated, and the subtleties of formal
calculus were quite involved.
\end{remark}

This paper is organized as follows. In \S2, we give some preliminary
results about ideals of vertex algebras. We also introduce the
notion of prime ideal, which is crucial for our proof. In \S3 we
generalize a classical result about commutative ring, of which the
main theorem comes as a corollary.

We assume that the readers are familiar with the axiomatic theory of
vertex algebra as introduced in \S3 and \S4 of \cite{ll}. Through
out this paper $(V,Y,1)$ will be a fixed vertex algebra.

\section{Some Preliminaries}
In this section, we present some preliminaries about ideals of
vertex algebras.

\begin{definition}
An ideal of the vertex algebra $(V,Y,1)$ is a subspace $I$ such that
for any $v\in V$ and $w\in I$,
$$Y(v,x)w\in I((x)),  \  \  Y(w,x)v\in I((x)).$$
\end{definition}

If $I$ is an ideal, we have $\mathcal {D}I\subset I$. We also have a
natural quotient vertex algebra $V/I$ together with the canonical
homomorphism $V\rightarrow V/I$.

\begin{definition}
For a subset $S$ of $V$, we define $(S)$ to be the smallest ideal
containing $S$ and we call $(S)$ the ideal generated by $S$. Then
$(S)$ is the intersection of all ideals of $V$ containing $S$. For
$x_{1},\cdots,x_{n}\in V$, we denote by $(x_{1},\cdots,x_{n})$ the
ideal generated by $\{x_{1},\cdots,x_{n}\}$. For an $F(x)\in
V[[x,x^{-1}]]$, we also denote by $(F(x))$ the ideal generated by
the set of coefficients of $F(x)$; this notion generalizes in an
obvious way to the case of several formal variables and the case of
several formal series.
\end{definition}

We will use two propositions in \S4.5 of \cite{ll} many times below;
we record them for the reader's convenience.

\begin{proposition}\label{p1}
Let $u,v,w\in V$, $p,q\in\mathbb{Z}$, then $u_{p}v_{q}w$ can be
expressed as a linear combination of elements of the form
$(u_{s}v)_{t}w$, with $s,t\in\mathbb{Z}$.
\end{proposition}

\begin{proposition}\label{p2}
Let $u,v,w\in V$, $p,q\in\mathbb{Z}$, then $u_{p}v_{q}w$ can be
expressed as a linear combination of elements of the form
$v_{s}u_{t}w$, with $s,t\in\mathbb{Z}$.
\end{proposition}

Combining proposition \ref{p1} with the $\mathcal {D}$-bracket
formulas for vertex algebras immediately implies:
\begin{corollary}
Let $S$ be a subset of $V$. Then $$(S)=spac\{v_{n}\mathcal
{D}^{i}(u)|v\in V, n\in\mathbb{Z}, i\geq 0 ,u\in S\}
$$$$=spac\{\mathcal {D}^{i}(v_{n}u)|v\in V, n\in\mathbb{Z}, i\geq 0
,u\in S\} $$$$=spac\{\mathcal {D}^{i}(v_{n}\mathcal {D}^{j}(u))|v\in
V, n\in\mathbb{Z}, i,j\geq 0 ,u\in S\}.
$$
\end{corollary}

Motivated by this corollary we introduce the following notation. Let
$\mathbbm{p},\mathbbm{q}$ be two subsets of $V$, set
$$\mathbbm{p}\mathbbm{q}=spac\{v_{n}\mathcal {D}^{i}(u)|v\in
\mathbbm{p}, n\in\mathbb{Z}, i\geq 0 ,u\in \mathbbm{p}\}
$$$$=spac\{\mathcal {D}^{i}(v_{n}u)|v\in \mathbbm{p}, n\in\mathbb{Z}, i\geq 0
,u\in \mathbbm{p}\} $$$$=spac\{\mathcal {D}^{i}(v_{n}\mathcal
{D}^{j}(u))|v\in \mathbbm{p}, n\in\mathbb{Z}, i,j\geq 0 ,u\in
\mathbbm{p}\}.
$$ The skew symmetry of vertex algebras immediately implies
$\mathbbm{p}\mathbbm{q}=\mathbbm{q}\mathbbm{p}$.

\begin{theorem}\label{th2}
For $u_{1},\cdots,u_{m},v_{1},\cdots,v_{n}\in V$, we have
$$(u_{1},\cdots,u_{m})(v_{1},\cdots,v_{n})=(\{Y(u_{i},x)v_{j}\}_{i,j}).$$
\end{theorem}

\begin{proof}
It suffices to prove that $(u)(v)=(Y(u,x)v)$ for any $u,v\in V$. The
inclusion $\supset$ follows directly from the definitions. For the
reverse inclusion, we have $$(u)(v)=span\{\mathcal
{D}^{i}(w_{l}u)_{m}w'_{n}v|l,m,n\in\mathbb{Z}, i\geq 0 ,w,w'\in V\}
$$$$
\subseteq span\{\mathcal
{D}^{i}w_{l}(w'_{m}u)_{n}v|l,m,n\in\mathbb{Z}, i\geq 0 ,w,w'\in V\}
$$$$
\subseteq span\{\mathcal
{D}^{i}w_{l}v_{m}w'_{n}u|l,m,n\in\mathbb{Z}, i\geq 0 ,w,w'\in V\}
$$$$
\subseteq span\{\mathcal
{D}^{i}w_{l}w'_{m}v_{n}u|l,m,n\in\mathbb{Z}, i\geq 0 ,w,w'\in V\}
$$$$
=(Y(v,x)u)=(Y(u,x)v),\ \ \ \ \ \ \ \ \ \ \ \ \ \ \ \ \ \ \  \ \ \ \
\ \ \ \ \ \
$$
where the first and the third $\subseteq$ follow from Proposition
\ref{p2}, the second $\subset$ and the last $=$ follow from the
property of skew symmetry of vertex algebras, and the second $=$
follows from Proposition \ref{p1}.
\end{proof}

\begin{definition}
An ideal I of $V$ is prime if $I\neq V$ and if $(Y(u,x)v)\subset I$
implies $u\in I$ or $v\in I$.
\end{definition}

We have, as in the classical case, the following proposition; the
proof is routine.
\begin{proposition}\label{p3}
Let $f:V\rightarrow V'$ be a surjective homomorphism of vertex
algebras, and let $I$ be an ideal of $V'$. Then $f^{-1}(I)$ is a
prime ideal of $V$ if and only if $I$ is a prime ideal of $V'$.
\end{proposition}

\section{The Nilpotent Elements}

\begin{definition}
An element $v$ of $V$ is said to be nilpotent if there exists a
positive integer $r$ such that $$Y(v,x_{1})\cdots Y(v,x_{r})=0.$$
\end{definition}

\begin{theorem}\label{th3}
The set $\mathfrak{R}$ of nilpotent elements $V$ is the intersection
of all the prime ideals of $V$.
\end{theorem}
This generalizes the corresponding result about commutative rings
\cite{am
}, the proof is exactly the same.

\begin{proof}
Denote by $\mathfrak{R}'$ the intersection of all the prime ideals
of $V$. We want to show that $\mathfrak{R}=\mathfrak{R}'$. The
inclusion $\subset$ follows immediately from the definition of prime
ideal.

In order to prove the reverse inclusion, suppose that $v\in V$ is
not nilpotent. Then using Proposition \ref{p1} repeatedly implies
$$Y(v,x_{1})\cdots Y(v,x_{n-1})v\neq 0$$ for each $n>0$.  Let
$\Gamma$ be the set of proper ideals $\mathbbm{m}$ satisfying
$$(Y(v,x_{1})\cdots Y(v,x_{n-1})v)\nsubseteq \mathbbm{m}$$
 for each
$n>0$. Order $\Gamma$ by inclusion, then Zorn's lemma implies
$\Gamma$ has a maximal element; let $\mathbbm{w} $ be a maximal
element of $\Gamma$. We shall show that $\mathbbm{w} $ is a prime
ideal. Suppose $a,b\not\in\mathbbm{w} $. Then the ideals
$\mathbbm{w}+(a) $, $\mathbbm{w}+(b) $ properly contain $\mathbbm{w}
$ and therefore do not lie in $\Gamma$. Thus there exist some
$m,n>0$ such that
$$(Y(v,x_{1})\cdots Y(v,x_{m-1})v)\subseteq \mathbbm{w}+(a),
$$$$(Y(v,x_{1})\cdots Y(v,x_{n-1})v)\subseteq \mathbbm{w}+(b).$$ Now
repeatedly using Proposition \ref{p1} yields
$$(Y(v,x_{1})\cdots Y(v,x_{m+n-1})v)\subseteq
(\mathbbm{w}+(a))(\mathbbm{w}+(b))=\mathbbm{w}+(Y(a,x)b),$$
therefore we have $(Y(a,x)b)\nsubseteq \mathbbm{w}$. Hence we find a
prime ideal $\mathbbm{w}$ that does not contain $v$. This finishes
the proof.
\end{proof}

Theorem \ref{th3} immediately implies that the set $\mathfrak{R}$ of
nilpotent elements $V$ forms an ideal. Hence we have the quotient
homomorphism $f:V\rightarrow V/\mathfrak{R}$. Now combining
Proposition \ref{p3} with Theorem \ref{th3} implies that
$V/\mathfrak{R}$ has no nonzero nilpotent elements. This finishes
the proof of Theorem \ref{th1}.

\end{document}